\documentclass[11pt,a4paper, reqno]{amsart}
\usepackage{amsmath,amsfonts,amssymb,mathrsfs}
\usepackage{mathtools}
\usepackage{graphicx}
\usepackage{comment}
\usepackage{xfrac}
\usepackage{subfigure}
\usepackage{bbm}
\usepackage{epstopdf}
\usepackage{pstricks}
\usepackage{pst-node}
\usepackage[linktocpage=true,colorlinks=true,linkcolor=blue,citecolor=red,urlcolor=magenta,pdfborder={0 0 0}]{hyperref}
\usepackage{fancyhdr}
\usepackage{mathscinet}
\usepackage{dutchcal}
\usepackage{xcolor}
\usepackage[english]{babel}
\usepackage{mathscinet}

\usepackage{chngcntr}

\usepackage{appendix}

\definecolor{forestgreen}{rgb}{0.13, 0.55, 0.13}
\definecolor{anna}{rgb}{0.01, 0.28, 1.0}

\newtheorem{theorem}{\bf Theorem}[section]
\newtheorem{lemma}[theorem]{\bf Lemma}

\newtheorem{definition}[theorem]{\bf Definition}
\newtheorem{remark}[theorem]{\bf Remark}

\def \GK {{\Gamma_{\!\! K}}}
\def \GL {{\Gamma_{\!\! L}}}
\newcommand{\R}{\mathbb{R}}

\newcommand{\N}{\mathbb{N}}

\newcommand{\OO}{\mathbb{O}}
\newcommand{\I}{\mathbb{I}}

\newcommand{\W}{\mathcal{W}}

\def \rnn {{\mathbb {R}}^{2n+1}}

\def \L {\mathscr{L}}
\def \K {\mathscr{K}}

\def \a {{\alpha}}
\def \b {{\beta}}

\def \g {{\gamma}}
\def \d {{\delta}}
\def \e {{\varepsilon}}

\def \epsilon {{\varepsilon}}

\def \k {{\kappa}}
\def \l {{\lambda}}
\def \r {{\rho}}
\def \s {{\sigma}}
\def \t {{\tau}}
\def \m {{\mu}}

\def \x {{\xi}}

\def \z {{\zeta}}

\def \phi {{\varphi}}

\def \G {{\Gamma}}
\def \O {{\Omega}}

\def \div {{\text{\rm div}}}
\def \loc {{\text{\rm loc}}}
\def \trace {{\text{\rm tr}}}
\def \tr {{\text{\rm Tr}}}
\def \diag {{\text{\rm diag}}}
\def \meas {{\text{\rm meas}}}
\def\p{\partial}
\def \tilde {\widetilde}

\def \Q {{\mathcal{Q}}}

\def \D {{\mathcal{D}}}

\usepackage{mathtools}

\begin{document}
	\title[On the fundamental solution]{On the fundamental solution for degenerate 
	Kolmogorov equations with rough coefficients}
	
	\author{Francesca Anceschi}
	\address{Dipartimento di Matematica e Applicazioni
		"Renato Caccioppoli" -
		Università degli Studi di Napoli "Federico II": 
		Via Cintia, Monte S. Angelo
		I-80126 Napoli, Italy}
	\email{francesca.anceschi@unina.it}
	
	
	\author{Annalaura Rebucci}
	\address{Dipartimento di Scienze Matematiche, Fisiche e Informatiche - 
	Università degli Studi di Parma: Parco Area delle Scienze, 7/A 43124 Parma, Italy}
	\email{annalaura.rebucci@unipr.it}
	
	\date{\today}

	\begin{abstract}
		\noindent
		The aim of this work is to prove the existence of a fundamental solution associated to the 
		Kolmogorov equation $\L u = f$ with measurable coefficients in the dilation invariant case. Moreover,
		we prove Gaussian upper and lower bounds for it, and other related properties. 
				
		
	\end{abstract}
	
	\maketitle
	
	\hypersetup{bookmarksdepth=2}
	\setcounter{tocdepth}{1}
	
	\tableofcontents
	
	\setcounter{equation}{0}\setcounter{theorem}{0}
	\section{Introduction}
	The aim of this work is to prove the existence of a fundamental solution for a 
	second order partial differential equation of Kolmogorov type with measurable coefficients of the form
	\begin{align}\label{defL}
		\L u (x,t) &:=\sum \limits_{i,j=1}^{m_0}\partial_{x_i}\left(a_{ij}(x,t)\partial_{x_j}u(x,t)\right) + 
		\sum \limits_{i=1}^{m_0} b_{i}(x,t) \partial_{x_i} u(x,t) + \\ \nonumber
		&+\sum \limits_{i,j=1}^N b_{ij}x_j\partial_{x_i}u(x,t)-\partial_t u(x,t)
		+c(x,t)u(x,t)=0,
	\end{align}
	where $z=(x,t)=(x_1,\ldots,x_N,t)\in \R^{N+1}$ and $1 \leq m_0 \leq N$. 		
	In particular, the matrices $A_0=(a_{ij}
	(x,t))_{i,j=1,\ldots,m_0}$ and $B=(b_{ij})_{i,j=1,\ldots,N}$ satisfy the following structural assumptions.
\medskip

\begin{itemize}
\item[\textbf{(H1)}] The matrix $A_0 $ is symmetric with real measurable entries, i.e. $a_{ij}(x,t)=a_{ji}(x,t)$, for every $i,j=1,\ldots,m_0$.
Moreover, there exist two positive constants $\lambda$ and $\Lambda$ such that 
\begin{eqnarray}\label{hypcost}
\lambda |\xi|^2 \leq \sum_{i,j=1}^{m_0}a_{ij}(x,t)\xi_i\xi_j \leq \Lambda|\xi|^2
\end{eqnarray}
for every $(x,t) \in \R^{N+1}$ and $\xi \in \R^{m_0}$. The matrix B has constant entries.
\end{itemize}

\medskip

\begin{itemize}
\item[\textbf{(H2)}] The \textit{principal part operator $\K$ of $\L$} is hypoelliptic, where $\K$ is defined as
\begin{eqnarray}\label{defK}
	\K u(x,t):=\sum \limits_{i=1}^{m_0}\partial^2_{x_i} u(x,t) + 
	\sum \limits_{i,j=1}^N b_{ij}x_j\partial_{x_i}u(x,t)-\partial_t u(x,t),
\end{eqnarray}
and it is dilation invariant with respect to the family of dilations $( \d_{r} )_{r>0}$ defined in \eqref{gdil}.
\end{itemize}

Note that we allow the operator $\L$ to be strongly degenerate whenever $m_0 < N$. However, it is known that the 
first order part of $\L$ may induce a strong regularizing property. Indeed, under suitable assumptions on the 
matrix $B$, the operator $\L$ is hypoelliptic, namely every distributional solution $u$ to $\L u = f$ defined in 
some open set $\Omega \subset \R^{N+1}$ belongs to $C^\infty(\Omega)$ and it is a classical solution to $\L u = f$, 
whenever $f \in C^\infty(\Omega)$. We refer to Section \ref{preliminaries} for additional information on this fact.
Lastly, when $\L$ is uniformly parabolic (i.e. $m_0=N$ and $B\equiv \mathbb{O}$), the assumption $\textbf{(H2)}$ is trivially satisfied. 
Indeed, in this case the heat operator is the principal part operator $\K$. 

\medskip

In order to  expose our main results we first need to introduce some preliminary notation. 
From now on, we consider the strip $S_{T_0T_1} := (T_0, T_1) \times \R^{N}$, and in accordance with the scaling of the differential equation (see \eqref{gdil} below) 
we split the coordinate $x\in\R^N$ as
\begin{equation}\label{split.coord.RN}
	x=\big(x^{(0)},x^{(1)},\ldots, x^{(\kappa)}\big), \qquad x^{(0)}\!\in\R^{m_0}, \quad x^{(j)}\!\in\R^{m_j}, \quad j\in\{1,\ldots,\kappa\},
\end{equation}
where every $m_{j}$ is a positive integer such that
\begin{equation*}
		\sum \limits_{j=1}^{\kappa} m_j = N 
		\qquad \text{and} \qquad N \ge m_0 \ge m_1
		\ge \ldots \ge m_\k \ge 1  .
\end{equation*}
Thus, here and in the sequel we denote 
\begin{eqnarray*}
	D=(\partial_{x_1},\ldots,\partial_{x_N}),\quad D_{m_0}=(\partial_{x_1},\ldots,\partial_{x_{m_0}}), \quad \langle\cdot,\cdot 
	\rangle, \quad \div,
\end{eqnarray*}
the gradient, the partial gradient in the first $m_0$ components, the inner product and the divergence in $R^N$, respectively. 
Moreover, we introduce the matrix 
\begin{eqnarray*}
A(x,t)=\left(a_{ij}(x,t)\right)_{1\leq i,j \leq N},
\end{eqnarray*}
where $a_{ij}$, for every $i,j=1,\ldots,m_0$, are the coefficients appearing in \eqref{defL}, while $a_{ij}\equiv 0$ whenever $i > m_0$ or $j>m_0$, and we let
\begin{align}  \label{drift}
Y:=\sum_{i,j=1}^N b_{ij}x_j\partial_{x_i}u(x,t)-\partial_t u(x,t) \quad \text{and} \quad
 b := (b_{1}, \ldots, b_{m_{0}}, 0, \ldots, 0).
\end{align}
Now, we are in a position to rewrite the operator $\L$ in the following compact form
\begin{eqnarray*}
\L u=\div(A D u) + Yu + \langle b, Du \rangle +  cu
\end{eqnarray*}
and we recall that its formal adjoint is defined as  
\begin{align} \label{Ladj}
	\L ^{*} v (\x, \t)= \sum_{i,j=1}^{m_0} \partial_{\xi_i} & \left( a_{ij} (\xi,\tau) \partial_{\xi_j} v (\x, \t) \right)
	 - \sum_{i=1}^{m_0} 
	 \partial_{\xi_i}(b_i(\xi,\tau) v(\x, \t)) \\ \nonumber
	 &+(c-\tr (B))v (\x, \t) +Y^{*}v(\xi,\tau)
\end{align} 
where 
\begin{equation*}
Y^{*}v(\xi,\tau) := - \sum \limits_{i,j=1}^N b_{ij}\xi_j\partial_{\xi_i}v (\x, \t) + \p_{\tau} v (\x, \t).
\end{equation*}

We denote by $\D(S_{T_0T_1})$ the set of $C^\infty$ functions compactly supported in $S_{T_0T_1}$.
From now on, $H^{1}_{x^{(0)}}$ denotes the Sobolev space of functions $u \in  L^{2} (\O_{m_{0}})$ with 
distributional gradient $D_{m_0}u$ lying in $( L^{2} (\O_{m_{0}}) )^{m_{0}}$, i.e. 
\begin{equation*}
	H^{1}_{x^{(0)}} := \left\{ u \in L^{2} (\O_{m_{0}}) : \, D_{m_{0}} u \in ( L^{2} (\O_{m_{0}}) )^{m_{0}}
	\right\},
\end{equation*}
and we set 
\begin{equation*}
	\| u \|_{H^{1}_{x^{(0)}}} := \| u \|_{L^{2} (\O_{m_{0}}) } + \| D_{m_{0} } u \|_{L^{2} (\O_{m_{0}})}.
\end{equation*}
In a standard manner, see for instance \cite{Brezis}, $\W$ denotes the closure of $C^{\infty} (\overline S_{T_0T_1})$ 
in the norm 
\begin{equation}
	\label{normW}
	\| u \|^2_{\W} = \| u \|^2_{L^2(\O_{N-m_{0} + 1};H^1_{x^{(0)}})} + \| Y u \|^2_{L^2(\O_{N-m_{0} + 1};
	H^{-1}_{x^{(0)}})},
\end{equation}
that is explicitly computed as follows:
\begin{align*}
	\| u \|^2_{\W} = \int_{\O_{N-m_{0} + 1}} \| u(\cdot, y,t) \|_{H^1_{x^{(0)}}}^{2} dy \, dt  + 
				 \int_{\O_{N-m_{0} + 1}}  \| Y u (\cdot, y,t) \|^2_{H^{-1}_{x^{(0)}}} dy \, dt ,
\end{align*}
where $y = (x^{(1)},\ldots, x^{(\kappa)})$. In particular, $\W$ is a Banach space,
and it was firstly introduced in \cite{AR-harnack} as an extension of the natural functional setting that arises in the study of the weak regularity theory for the 
kinetic Kolmogorov-Fokker-Planck equation \cite{AM, GIMV, GI, GM}.
From now on, we consider the shorthand notation $L^{2}H^{-1}$ to denote $L^{2}(\O_{N-m_{0}+1}; H^{-1}_{c,x^{(0)}})$.
Now, we introduce the definition of weak solution we consider in our work. 
\begin{definition}\label{weak-sol2}
A function $u \in \W$ is a weak solution to \eqref{defL} if for every non-negative test function $\phi \in \D(S_{T_0T_1})$, we have
\begin{align}\label{kolmo}
	   \int_{S_{T_0T_1}} - \langle A Du, D\phi \rangle + \phi Y u + \langle b , Du \rangle \phi + c u \phi= 0.
	\end{align}
In the sequel, we will also consider weak sub-solutions to \eqref{defL}, namely functions $u \in \W$ that satisfy the following inequality
\begin{align}\label{kolmo-sub}
	   \int_{S_{T_0T_1}} - \langle A Du, D\phi \rangle + \phi Y u + \langle b , Du \rangle \phi + c u \phi \geq 0,
	\end{align}
	for every non-negative test function $\phi \in \D(S_{T_0T_1})$. A function $u$ is a super-solution to \eqref{defL} if $-u$ 
	is a sub-solution.
\end{definition}
Finally, we recall the definition of weak fundamental solution for the operator $\L$, firstly introduced by Lanconelli, 
Pascucci and Polidoro in \cite[Definition 2.2]{LPP}.
\begin{definition}
	\label{funsolk}
	A  {\rm weak} fundamental solution for $\L$ is a continuous and positive function 
	$\GL = \GL(x,t; x_0, t_0)$ defined for $t \in \R$, $0 \le T_0 < t_0 < t < T_1$ and any $x, x_0 \in \R^{N}$ such that:
	\begin{enumerate}
		\item $\GL = \GL(\cdot, \cdot; x_0,t_0)$ is a weak solution to 
			$\L u = 0$ in $S_{T_0T_1}$ and $\GL = \GL(x,t; \cdot, \cdot)$ is a weak solution of $\L^* u = 0$ 
			in $S_{T_0T_1}$;
		\item for any bounded function $\phi \in C (\R^{N})$ and any $x,x_0 \in \R^{N}$ we have
			\begin{align} \label{uv-def}
				&\begin{cases}
					\L u (x,t) = 0 \qquad &(x,t) \in (T_0, T_1) \times \R^N, \\
					\lim \limits_{(x,t) \to (x_0, t_0) \atop t >t_0} u(x,t) = \phi(x_0) \qquad &x_0 \in \R^N,
				\end{cases}
				\\ \nonumber
				&\begin{cases} 
					\L^* v (x_0,t_0) = 0 \qquad &(x_0,t_0) \in (T_0, T_1) \times \R^N, \\
					 \lim \limits_{(x_0,t_0) \to (x,t) \atop t< t_0} v(x_0,t_0) = \phi(x)  \qquad &x \in \R^N ,  
				\end{cases}
			\end{align}
			where
			\begin{align} \label{rep1}
				u(x,t) := \int_{\R^{N}} \GL(x,t;x_0,t_0) \, \phi(x_0) \, dx_0 , \qquad
				v(x_0,t_0) := \int_{\R^{N}}\GL(x,t;x_0,t_0) \, \phi(x) \, dx .
			\end{align}
	\end{enumerate}
\end{definition}

Now, we are in a position to state our main results. Firstly, we give answer to \cite[Remark 2.3]{LPP} by 
proving the existence of a weak fundamental solution for the operator $\L$ in the sense of Definition \ref{funsolk}
under the following assumption for the first order coefficients $b$ and $c$.
\begin{itemize}
\item[\textbf{(H3A)}] The lower order coefficients $b_i$, with $i=1,\ldots,m_0$, and $c$ are bounded measurable functions of 
		$(x,t) \in \R^{N+1}$, i.e. there exists a positive constant $M$ such that 
		\begin{equation*}
|b(x,t)| \leq M, \quad |c(x,t)| \leq M, \quad \forall (x,t) \in \R^{N+1}.		
		\end{equation*}
\end{itemize}

\begin{theorem}[Existence of the weak fundamental solution]
	\label{main1}
	Let us consider operator $\L$ under the assumptions \textbf{(H1)}-\textbf{(H3A)}. Then there exists a 
	fundamental solution $\GL$ of $\L$ in the sense of Definition \ref{funsolk} and the following properties hold:
	\begin{enumerate}
		\item the reproduction property holds. Indeed, 
		for every $x, y \in \R^{N}$ and $t,\tau \in \R$ with $t_0 < \t < t$ such that 
		$t_0, t \in (T_0, T_1)$: 
		\begin{equation*}
			\GL(x,t; y,t_0) = \int \limits_{\R^{N}} \GL(x,t;\x,\t) \, \GL(\x, \t; y,t_0) \, d\x;
		\end{equation*}
		\item for every $(x,t), (y,t_0) \in \R^{N+1}$ with $t \le t_0$
		we have that $\GL (x,t; y,t_0) = 0$.
	\end{enumerate} 
	Moreover, the function $\GL^{\!\!*}(x, t; y,t_0) = \GL (y,t_0; x,t)$ is the fundamental solution of $\L^*$ and verifies the 
	dual properties of this statement.
\end{theorem}
To our knowledge 
this is the first existence result available for the weak fundamental solution
to \eqref{defL} in the sense of the above definition.
Moreover, we emphasize that the PDE approach adopted in this work improves the previously known results in that it allows us 
to consider differential operators with \emph{only measurable} coefficients, which is a milder assumption than the 
usual ones, see \cite{LP, AMP-existence, BP}.

Secondly, we extend \cite[Theorem 1.3]{LPP} providing Gaussian upper and lower bounds for the weak fundamental solution 
$\Gamma$ of $\L$ under the following more general assumption on the lower order coefficients $b$ and $c$. 

\medskip

\begin{itemize}
\item[\textbf{(H3B)}] The coefficients $b \in \left( L^q_{\loc} (S_{T_0T_1}) \right)^{m_0}$, $c \in L^q_{\loc}
	(S_{T_0T_1})$ for some $q >\frac34(Q+2)$. Moreover, $\div b \ge 0$ and $c \le 0$.
\end{itemize}
\begin{remark}	
	As pointed out in \cite[Remark 1.7]{AR-harnack}, we can replace assumption \textbf{(H3B)}  
	with the one firstly considered by Wang and Zhang in \cite{WZ3, WZ-preprint}:
	$b\in L^q_{loc}(S_{T_0T_1})$ for some $q >(Q+2)$ and $c \in L^q_{loc}(S_{T_0T_1})$ for some $q >\frac{Q+2}{2}$,
	with the additional requirement of $c \le 0$. Note that the last assumption on the sign of $c$ is necessary to handle 
	unbounded coefficients, and thus when working with \textbf{(H3A)} we are able to drop it.
\end{remark}

A first result regarding Gaussian upper bounds independent of the H\"older norm of the coefficients is due to Pascucci and Polidoro, who studied operator \eqref{defL} with $b=c=0$ (see \cite[Theorem 1.1]{PP-upper}). 
Later on, Lanconelli, Pascucci and Polidoro \cite[Theorem 4.1]{LPP} extended the Aronson procedure to 
\eqref{defL} with bounded lower order coefficients. 

On the other hand, if we consider Gaussian lower bounds independent of the H\"older norm of the coefficients of operator $\L$, a first result is due to 
Lanconelli, Pascucci and Polidoro \cite[Theorem 1.3]{LPP} for the particular case of the kinetic Kolmogorov-Fokker-Planck equation. The proof of this
result is based on the construction of a Harnack chain starting from a Harnack inequality of the some kind of Theorem \ref{harnack-thm}, 
alongside with the study of the control problem associated to the principal part operator $\K$. 

\begin{theorem}[Gaussian bounds] 
	\label{gauss-bound-K}
	Let $\L$ be an operator of the form \eqref{defL} under the assumptions \textbf{(H1)}-\textbf{(H3B)}.
	Let $I=(T_{0}, T_{1})$ be a bounded interval, then there exist four positive constants 
	$\l^{+}$, $\l^{-}$, $C^{+}$, $C^{-}$ such that 
       \begin{align}\label{gaussboundsform}
		C^{-} \, \GK^{\! \! \! \! \l^{-}} (x,t; y,T) \, \le \, \GL(x,t; y,T) \, \le \,
		C^{+} \, \GK^{\! \! \! \! \l^{+}} (x,t;y,T)
	\end{align}
	for every $(x,t), (y,T) \in \R^{N+1}$ with 
	$T_{0} <  T < t < T_{1}$.
	The constants $\l^{+}$, $\l^{-}$, $C^{+}$, $C^{-}$ only depend on $B$, $(T_{1}-T_{0})$, $\| b \|_q$ and $\| c \|_q$. Note that
	$\GK^{\! \! \! \! \l^{-}}$ and $\GK^{\! \! \! \! \l^{+}}$ respectively denote the fundamental solution of 
	$\K_{\l^{-}}$ and $\K_{\l^{+}}$, respectively, where
	\begin{equation} \label{Klambda}
		\K^{\l} u(x,t):= \frac{\l}{2} \sum \limits_{i=1}^{m_{0}} \p_{x_{i}}^{2} u(x,t) + 
		\sum_{i,j=1}^N b_{ij}x_j\partial_{x_i}u(x,t)-\partial_t u(x,t),
	\end{equation}
	and the explicit expression of $\GK^{\! \! \! \! \l^{\pm}}$ is 
	\begin{equation} \label{Glambda} 
		\GK^{\! \! \! \! \l} ( x,t;0,0) = \begin{cases}
		\frac{(2 \pi \l)^{-\frac{N}{2}}}{\sqrt{\text{det} C(t)}} \exp \left( - 
		\frac{1}{2\l} \langle C^{-1} (t) x, x \rangle - t \, \trace (B) \right), \hspace{3mm} & 
		\text{if} \hspace{1mm} t > 0, \\
		0, & \text{if} \hspace{1mm} t \le 0.
	\end{cases}
	\end{equation}
\end{theorem}



\begin{remark}
Theorem \ref{gauss-bound-K} holds true in a more general setting than Theorem \ref{main1}. 
Moreover, since the proof of the upper bound in \eqref{gaussboundsform} does not rely on the Harnack inequality, the 
rightmost inequality of \eqref{gaussboundsform} holds true for the more general operator 
\begin{align}\label{Ltilde}
		\tilde{\L} u (x,t) &:=\sum \limits_{i,j=1}^{m_0}\partial_{x_i}\left(a_{ij}(x,t)\partial_{x_j}u(x,t)\right) + 
		\sum \limits_{i=1}^{m_0} b_{i}(x,t) \partial_{x_i} u(x,t) + \\ \nonumber
		&-\sum \limits_{i=1}^{m_0}\p_{x_i}(a_i(x,t)))+\sum \limits_{i,j=1}^N b_{ij}x_j\partial_{x_i}u(x,t)-\partial_t u(x,t)
		+c(x,t)u(x,t),
	\end{align}
with $a \in \left( L^q_{\loc} (S_T) \right)^{m_0}$ and $\div \,a \ge 0 $.  Finally, we observe that, when dealing with bounded first order coefficients, the costants $\l^{+}$, $\l^{-}$, $C^{+}$, $C^{-}$ appearing in \eqref{gaussboundsform} clearly do not depend on $\Vert b \Vert_q$ and $\Vert c \Vert_q$.
\end{remark}

\subsection{Motivation and background}  
Kolmogorov equations appear in the theory of stochastic processes as linear second order parabolic equations with non-negative characteristic form. 
In its simplest form, if $\left(W_t \right)_{t \ge 0}$ denotes a real Brownian motion, the density $p=p(t,v,y,v_{0},y_{0})$ of the stochastic process  $(V_{t}, Y_{t})_{t \ge 0}$
\begin{equation}
   \label{processo}
    \begin{cases}
    V_t = v_{0} + \s W_{t} \\
    Y_t = y_{0} + \int_{0}^{t} V_{s} \, ds
    \end{cases}
\end{equation}
is a solution to one of the simplest strongly degenerate Kolmogorov equation, that is
\begin{equation}
	\label{sdk}
        \tfrac12 \s^{2} \p_{vv} p + v \p_y p = \p_t p,  \qquad t \ge 0, \qquad
    (v,y) \in \R^2.
\end{equation}
In 1934 Kolmogorov provided us with the explicit expression of the density $p=p(t,v,y,v_{0},y_{0})$ of the above equation (see \cite{K1})
\begin{equation}
   \label{sf}
   p(t,v,y, v_{0}, y_{0}) = \tfrac{\sqrt{3}}{2 \pi t^2} 
    \exp \left( - \tfrac{(v-v_{0})^2}{t} - 3 \tfrac{(v-v_{0})(y - y_{0} - t v_{0})}{t^2} - 3 \tfrac{(y - y_{0} - t y_{0})^2}{t^3} \right) \quad t > 0,
\end{equation}
and pointed out it is a smooth function despite the strong degeneracy of \eqref{sdk}. 
This immediately suggested that the operator $\L$ associated to equation \eqref{sdk}
\begin{equation} \label{KolmoR3}
	\L := \tfrac12 \s^{2} \p_{vv} + v \p_y - \p_t, 
\end{equation}
is hypoelliptic. Indeed, later on H\"ormander considered this operator as a prototype for the family of hypoelliptic operators studied in his seminal work \cite{H}. 

Kolmogorov equations find their application in different research fields. First of all, the process in \eqref{processo} is the solution to the Langevin equation
\begin{equation*}
    \begin{cases}
    d V_t = d W_{t} \\
    d Y_t =  V_{t} \, dt,
    \end{cases}
\end{equation*}
and therefore Kolmogorov equations are related to every stochastic process satisfying Langevin equation. In particular, several 
mathematical models involving linear and non linear Kolmogorov type equations have also appeared in finance \cite{AD}, 
\cite{B}, \cite{BP} and \cite{DHW}. Indeed, equations of the form \eqref{sdk} appear in various models for pricing of 
path-dependent financial instruments (cf., for instance, \cite{BPVE} \cite{PA}), where, for example the equation
\begin{equation}
    \label{bpv}
     \p_t P + \tfrac12 \s^{2} S^2 \p^2_S P + (\log S ) \p_{A} P + r (S \p_{S} P - P) = 0, \qquad S > 0, 
    \, A, t \in \R
\end{equation}
arises in the Black and Scholes option pricing problem
\begin{equation*}
    \begin{cases}
    d S_t = \m S_{t} dt + \s S_{t} d W_{t} \\
    d A_t =  S_{t} \, dt,
    \end{cases}
\end{equation*}
where $\s$ is the volatility of the stock price $S$, $r$ is the interest rate of a risckless bond and $P= P(S, A, t)$ is the price of the Asian option depending on the price of the stock $S$, the geometric average $A$ of the
past price and the time to maturity $t$. In this framework, knowing that the fundamental solution to the Kolmogorov equation exists is helpful for the study of the option pricing problem and allows us to have various advantages when dealing with numerical simulations. For further information on this topic, we refer to \cite{AMP-existence, BPV}.
 
Moreover, we recall that the Kolmogorov equation is the prototype for a family of evolution equations arising in kinetic theory of gas, which take the following general form 
\begin{equation}
    \label{kt}
    Y u = \mathcal{J} (u).
\end{equation}
In this case, we have that $u=u(v,y,t)$ is the density 
of particles with velocity $v=(v_{1}, \ldots, v_{n})$ and position $y=(y_{1}, \ldots, y_{n})$ at time $t$. 
Moreover, 
\begin{equation*}
    Y u := \sum \limits_{j=1}^n v_j \p_{y_{j}} u + \p_t u
\end{equation*}
is the so called total derivative with respect to time in the phase space 
$\rnn$, and $\mathcal{J} (u)$ is the collision operator, which can be either 
linear or non-linear. For instance, in the usual Fokker-Planck 
equation (cf. \cite{DV}, \cite{R}) we have a linear collision operator  of the form
\begin{equation*}
    \mathcal{J} (u) = \sum \limits_{i,j=1}^n a_{ij} \, \p_{v_{i}, v_{j}}^2 u + \sum 
    \limits_{i=1}^n a_i \, \p_{v_{i}} u + a u
\end{equation*}
where $a_{ij}$, $a_i$ and $a$ are functions of $(y,t)$; $\mathcal{J} (u)$ can
also occur in divergence form
\begin{equation*}
     \mathcal{J} (u) = \sum \limits_{i,j=1}^n \p_{v_i} (a_{ij} \, \p_{v_{j}} u + 
     b_i u ) + \sum \limits_{i=1}^n a_{i} \p_{v_{i}} u + a u.
\end{equation*}
We also mention the following non-linear collision operator of the 
Fokker-Planck-Landau type
\begin{equation*}
    \mathcal{J} (u) = \sum \limits_{i,j=1}^n \p_{v_i} \big(a_{ij} (z,u) \p_{v_{j}} 
    u + b_i(z,u) \big), 
\end{equation*}
where the coefficients $a_{ij}$ and $b_i$ depend both on $z \in \rnn$ and the unknown functions $u$ through some integral expression. 
Moreover, this last operator is studied as a simplified version of the Boltzmann collision operator (see for instance \cite{C}, \cite{L}). For the description 
of wide classes of stochastic processes and kinetic models leading to equations of the previous type, we refer to the classical monographies \cite{C}, \cite{CC} and \cite{DM}.

\subsection{Plan of the paper}
This work is organized as follows. In Section \ref{preliminaries} we recall the properties of the geometrical structure associated to operator $\L$. In Section \ref{proof1} we prove Gaussian lower 
bounds for the fundamental solution associated to operator $\L$ under the assumption \textbf{(H3B)}. In Section \ref{proof2} we prove the existence of a weak fundamental solution for operator $\L$ under the assumption \textbf{(H3A)}.

\subsection*{Acknowledgements}
The first author is funded by the research grant PRIN2017 2017AYM8XW ``Nonlinear Differential Problems via 
Variational, Topological and Set-valued Methods''. 

\setcounter{equation}{0}\setcounter{theorem}{0}
\section{Preliminaries}
\label{preliminaries}
In this section we recall notation and known results about the non-Euclidean 
geometry underlying the operators $\L$ and $\K$. We refer to the survey paper \cite{APsurvey} and the references 
therein for a comprehensive treatment of this subject.

As first observed by Lanconelli and Polidoro in \cite{LP}, the principal part operator $\K$ is invariant with respect to left translations in the group $\mathbb{K}=(\mathbb{R}^{N+1},\circ)$, where the group law is defined by 
\begin{equation}
	\label{grouplaw}
	(x,t) \circ (\xi, \tau) = (\xi + E(\tau) x, t + \tau ), \hspace{5mm} (x,t),
	(\xi, \tau) \in \R^{N+1},
\end{equation}
and
\begin{equation}\label{exp}
	E(s) = \exp (-s B), \qquad s \in \R.
\end{equation} 
Then $\mathbb{K}$ is a non-commutative group with zero element $(0,0)$ and inverse
\begin{equation*}
(x,t)^{-1} = (-E(-t)x,-t).
\end{equation*}
For a given $\zeta \in \R^{N+1}$ we denote by $\ell_{\z}$ the left traslation on $\mathbb{K}=(\R^{N+1},\circ)$ defined as follows
\begin{equation*}
	\ell_{\z}: \R^{N+1} \rightarrow \R^{N+1}, \quad \ell_{\z} (z) = \z \circ z.
\end{equation*}
Then the operator $\K$ is left invariant with respect to the Lie product $\circ$, that is
\begin{equation*}
	\label{ell}
    \K \circ \ell_{\z} = \ell_{\z} \circ \K \qquad {\rm \textit{or, equivalently,}} 
    \qquad \K\left( u( \z \circ z) \right)  = \left( \K u \right) \left( \z \circ z \right),
\end{equation*}
for every $u$ sufficiently smooth.

We recall that, by \cite{LP} (Propositions 2.1 and 2.2), assumption \textbf{(H2)}
is equivalent to assume that, for some basis on $\R^N$, the matrix $B$ takes the following form
\begin{equation}
	\label{B}
	B =
	\begin{pmatrix}
		\OO   &   \OO   & \ldots &    \OO   &   \OO   \\  
		B_1   &    \OO  & \ldots &    \OO    &   \OO  \\
		\OO    &    B_2  & \ldots &  \OO    &   \OO   \\
		\vdots & \vdots & \ddots & \vdots & \vdots \\
		\OO    &  \OO    &    \ldots & B_\k    & \OO
	\end{pmatrix}
\end{equation}
where every $B_j$ is a $m_{j} \times m_{j-1}$ matrix of rank $m_j$, $j = 1, 2, \ldots, \k$ 
with 
\begin{equation*}
	m_0 \ge m_1 \ge \ldots \ge m_\k \ge 1 \hspace{5mm} \text{and} \hspace{5mm} 
	\sum \limits_{j=0}^\k m_j = N.
\end{equation*} 
In the sequel we will assume that $B$ has the canonical form \eqref{B}.
We remark that assumption \textbf{(H2)} is implied by the condition introduced by H\"{o}rmander in \cite{H}
applied to $\K$:
\begin{equation}\label{e-Horm}
{\rm rank\ Lie}\left(\partial_{x_1},\dots,\partial_{x_{m_0}},Y\right)(x,t) =N+1,\qquad \forall \, (x,t) \in \R^{N+1},
\end{equation}
where ${\rm  Lie}\left(\partial_{x_1},\dots,\partial_{x_{m_0}},Y\right)$ denotes the Lie algebra generated by the first order differential operators $\left(\partial_{x_1},\dots,\partial_{x_{m_0}},Y\right)$ computed at $(x,t)$.
Yet another condition equivalent to {\bf (H2)}, (see \cite{LP}, Proposition A.1), is that
\begin{eqnarray}\label{Cpositive}
C(t) > 0, \quad \text{for every $t>0$},
\end{eqnarray}
where
\begin{equation}\label{defC}
	C(t) = \int_0^t \hspace{1mm} E(s) \, A_0 \, E^T(s) \, ds,
\end{equation}
and $E(\cdot)$ is the matrix defined in \eqref{exp}.
Lastly, we recall that H\"{o}rmander explicitely constructed in \cite{H} the fundamental solution of $\K$ as
\begin{equation} \label{eq-Gamma0}
	\GK (z,\zeta) = \GK(\zeta^{-1} \circ z, 0 ), \hspace{4mm} 
	\forall z, \zeta \in \R^{N+1}, \hspace{1mm} z \ne \zeta,
\end{equation}
where
\begin{equation} \label{eq-Gamma0-b}
	\GK ( (x,t), (0,0)) = \begin{cases}
		\frac{(4 \pi)^{-\frac{N}{2}}}{\sqrt{\text{det} C(t)}} \exp \left( - 
		\frac{1}{4} \langle C^{-1} (t) x, x \rangle - t \, \trace (B) \right), \hspace{3mm} & 
		\text{if} \hspace{1mm} t > 0, \\
		0, & \text{if} \hspace{1mm} t \le 0.
	\end{cases}
\end{equation}
In particular, condition \eqref{Cpositive} implies that $\GK$ in \eqref{eq-Gamma0-b} is well-defined.

Let us now consider the second part of assumption \textbf{(H2)}. 
We say that $\K$ is invariant with respect to $(\delta_r)_{r>0}$ if  
\begin{equation}
	\label{Ginv}
      	 \K \left( u \circ \delta_r \right) = r^2 \delta_r \left( \K u \right), \quad \text{for every} \quad r>0,
\end{equation}
for every function $u$ sufficiently smooth. It is known (see Proposition 2.2 of \cite{LP}) that it is possible to read this dilation invariance property in the expression of the matrix $B$ in \eqref{B}. More precisely, $\K$ satisfies \eqref{Ginv} if and only if the matrix $B$ takes the form \eqref{B}.
In this case, we have 
\begin{equation}
	\label{gdil}
	\d_{r} = (\d^0_r,r^2), \qquad \qquad r > 0,
\end{equation}
where 
\begin{equation}
	\label{gdil0}
	\d^0_{r} = \text{diag} ( r \I_{m_0}, r^3 \I_{m_1}, \ldots, r^{2\k+1} \I_{m_\k}), \qquad \qquad r > 0.
\end{equation}
Since $\K$ is dilation invariant with respect to $(\d_r )_{r>0}$, also its fundamental solution $\GK$ is a homogeneous function of degree $- Q$, namely
\begin{equation*}
	\GK \left( \d_{r}(z), 0 \right) = r^{-Q} \hspace{1mm} \GK
	\left( z, 0 \right), \hspace{5mm} \forall z \in \R^{N+1} \setminus
	\{ 0 \}, \hspace{1mm} r > 0.
\end{equation*}
We next introduce a homogeneous norm of degree $1$ with respect to the dilations 
$(\d_{r})_{r>0}$ and a corresponding quasi-distance which is invariant with respect to the group operation \eqref{grouplaw}.
\begin{definition}[Homogeneous norm]
	\label{hom-norm}
	Let $\a_1, \ldots, \a_N$ be positive integers such that
	\begin{equation}\label{alphaj}
		\diag \left( r^{\a_1}, \ldots, r^{\a_N}, r^2 \right) = \d_r.
	\end{equation}
	If $\Vert z \Vert = 0$ we set $z=0$ while, if $z \in \R^{N+1} 
	\setminus \{ 0 \}$ we define $\Vert z \Vert = r$ where $r$ is the 
	unique positive solution to the equation
	\begin{equation*}
		\frac{x_1^2}{r^{2 \a_1}} + \frac{x_2^2}{r^{2 \a_2}} + \ldots 
		+ \frac{x_N^2}{r^{2 \a_N}} + \frac{t^2}{r^4} = 1.
	\end{equation*}
	Accordingly, we define the quasi-distance $d$ by
	\begin{equation}\label{def-dist}
		d(z, w) = \Vert z^{-1} \circ w \Vert , \hspace{5mm} z, w \in 
		\R^{N+1}.
	\end{equation}
\end{definition}
	As $\det E(t) = e^{t \hspace{1mm} \text{\rm trace} \,
		B} = 1$, the Lebesgue measure is invariant with respect to the translation group 
	associated to $\K$. Moreover, since $\det \d_r = r^{Q+2}$, we also have 
	\begin{equation*}
		\meas \left( \Q_r(z_0) \right) = r^{Q+2} \meas \left( \Q_1(z_0) \right), \qquad \forall
		\ r > 0, z_0 \in \R^{N+1},
	\end{equation*}
	where
	\begin{equation}
		\label{hom-dim}
		Q = m_0 + 3 m_1  + \ldots + (2\k+1) m_\k.
	\end{equation}
	The natural number $Q+2$ is called the \textit{homogeneous dimension of} $\R^{N+1}$ \textit{with respect to} 
	$(\d_{r})_{r > 0}$.
	This denomination is proper since the Jacobian determinant of $\d_{r}$ equals to $r^{Q+2}$.

	 We now recall the definition of H\"older continuous function in this framework.
\begin{definition}[H\"older continuity]
    \label{holdercontinuous}
    Let $\a$ be a positive constant, $\a \le 1$, and let $\O$ be an open subset of $\R^{N+1}$. We 
    say that a function $f : \O \longrightarrow \R$ is H\"older continuous with exponent $\a$ in $\O$
    with respect to the group $\mathbb{K}=(\mathbb{R}^{N+1},\circ)$, defined in \eqref{grouplaw}, (in short: H\"older 
    continuous with exponent $\a$, $f \in C^\a_{K} (\O)$) if there exists a positive constant $C>0$ such that 
    \begin{equation*}
        | f(z) - f(\z) | \le C \; d(z,\zeta)^{\a} \qquad { \rm for \, every \, } z, \z \in \O,
    \end{equation*}
where $d$ is the distance defined in \eqref{def-dist}. Moreover, to every bounded function $f \in C^\a_{K} (\O)$ we associate the semi-norm
    	\begin{equation*}
       		[ f ]_{C^{\a} (\O)} = 
       		\hspace{1mm} 
       		 \sup \limits_{z, \z \in \O \atop  z \ne \z} \frac{|f(z) - f(\z)|}{d(z,\zeta)^{\a}}.
        \end{equation*}
    Moreover, we say a function $f$ is locally H\"older continuous, and we write $f \in C^{\a}_{K,\loc}(\O)$,
    if $f \in C^{\a}_{K}(\O')$ for every compact subset $\O'$ of $\O$.
\end{definition}

Furthermore, we introduce the family of slanted cylinders on which we usually study the local
properties of the Kolmogorov equation starting from the unit past cylinder
\begin{eqnarray}\label{unitcylind}
\Q_1 := B_1 \times B_1 \times \ldots \times B_1 \times (-1,0),
\end{eqnarray}
defined through the open balls 
\begin{eqnarray}\label{openball}
B_1 = \lbrace x^{(j)}\!\in\R^{m_j} : \vert x \vert \leq 1 \rbrace,
\end{eqnarray}
where $j=0,\ldots,\kappa$ and $\vert \cdot \vert$ denotes the euclidean norm in $\R^{m_j}$. 
Now, for every $z_0 \in \mathbb{R}^{N+1}$ and $r>0$, we set
\begin{eqnarray}\label{rcylind}
\Q_r(z_0):=z_0\circ\left(\delta_r\left(\Q_1\right)\right)=\lbrace z \in \mathbb{R}^{N+1} \,:\, z=z_0\circ \delta_r(\zeta), \zeta \in \Q_1 \rbrace
\end{eqnarray}
the cylinder centered at an arbitrary point $z_{0} \in \R^{N+1}$ and of radius $r$.

We conclude this section by presenting an overview of results regarding the \textit{classical} theory and the corresponding definition of fundamental solution for the operator $\L$ under the 
following assumption on the coefficients $a$, $b$ and $c$.
\begin{itemize}
\item[\textbf{(C)}] The matrix $A_0$ satisfies assumption \textbf{(H1)}. Moreover, its coefficients are bounded and locally H\"older continuous of exponent $\a  \in (0,1]$, while 
			   the matrix $B$ has constant entries. The principal part operator $\K$ satisfies assumption \textbf{(H2)}.
			   The coefficients $ b_i$, for $i =1, \ldots, m_0$, and $c$ are bounded and H\"older continuous of exponent $\a \in (0,1]$.
\end{itemize}

First of all, let us recall the notion of Lie derivative $Y u$ of a function $u$ with respect to the vector field $Y$ defined in
\eqref{drift}. A function $u$ is Lie differentiable with respect to $Y$ at the point $(x,t)$ if there exists and is finite
\begin{equation} \label{lie-diff} 
		Yu(x,t) := \lim \limits_{s \rightarrow 0} \frac{u(\g(s)) - u(\g(0))}{s}, \qquad \g(s) = (E(-s) x, t - s).
\end{equation}
Note that $\g$ is the integral curve of $Y$, i.e. $\dot \g (s) = Y(\g(s))$. 
Clearly, if $u \in C^{1}(\O)$, with $\O$ open subset of $\R^{N+1}$, then $Y u (x,y,t)$ agrees with $\sum_{i,j=1}^N b_{ij}x_j\partial_{x_i}u(x,t)-\partial_t u(x,t) $ considered as a linear combination of the derivatives of $u$. Then we are in a position to introduce the notion of classical solution to
$\L u = 0$ under the assumptions \textbf{(C)}.

\begin{definition}
	\label{solution}
	A function $u$ is a solution to the equation $\L u = 0$ in a domain $\O$ of $\R^{N+1}$ under the assumptions \textbf{(C)}
	if the derivatives $\p_{x_i} u, \p_{x_i x_j}^{2} u$, for $i,j=1,\ldots, m_0$, and the Lie derivative $Yu$ exist as continuous functions in $\O$, 
	and the equation $\L u (x,t)= 0$
	is satisfied at any point $(x,t) \in \O$. Moreover, 
	we say that $u$ is a \emph{classical super-solution} 
	 to $\L u = 0$ if $\L u \le 0$.
	We say that $u$ is a \emph{classical sub-solution} if $-u$ is a classical supersolution.
\end{definition}
A fundamental tool in the classical regularity theory for Partial Differential Equations are Schauder estimates.
In particular, we recall the result proved by Manfredini in \cite{M} (see Theorem 1.4)  for classical solutions to 
$\L u = 0$, where the natural functional setting is
 \begin{equation*}
	C^{2 + \a} (\O) = \left\{ u \in C^{\a} (\O) \; \mid \; \p_{x}u , \p^2_{x} u, Y u \in C^{\a} (\O) \right\},
   \end{equation*}
and $C^{\a} (\O)$ is given in Definition \ref{holdercontinuous}. Moreover, if $u \in C^{2 + \a} (\O)$ then we define the norm
\begin{equation*}
	| u |_{2 + \a, \O } := | u |_{\a, \O } \; + \; | \p_{x}u |_{\a, \O } \; + \;  |\p^2_{x}u |_{\a, \O }  \; + \;  |Y u |_{\a, \O }.
   \end{equation*}
Clearly, the definition of $C^{2+\a}_{\loc} (\O)$ follows straightforwardly from the definition of $C^{\a}_{\loc} (\O)$. Finally, let us remark that we write $u \in C^2(\O)$ if $u$, its derivatives $\p_{x_i} u, \p_{x_i x_j}^{2} u$, for $i,j=1,\ldots, m_0$, and the Lie derivative $Yu$ exist as continuous functions in $\O$.

As we work with first order coefficients which are not H\"{o}lder continuous but only measurable, we now introduce the Schauder type estimates proved in \cite{PRS}.

First of all, we recall that the modulus of continuity of a function $f$ on any set $H \subset \R^{N+1}$ is defined as follows
\begin{equation} \label{eq-omegaf}
\omega_f(r) := \sup_{\substack{(x,t), (\x,\t)\in H \\ d((x,t) ,(\x,\t))<r}} |f(x,t)-f(\x, \t)|.
\end{equation}
Moreover, we recall the following definition.
\begin{definition} \label{d-DC}
A function $f$ is said to be Dini-continuous in $H$ if 
\begin{equation*}
\int_{0}^{1} \frac{\omega_f(r)}{r}dr < + \infty.
\end{equation*}
\end{definition}
We are now in position to state the following result (see \cite[Theorem 1.6]{PRS}).

\begin{theorem} \label{th-1}
Let $\L$ be an operator in the form \eqref{defL} satisfying hypothesis {\rm \bf(H1)-(H3A)}. Let $
u$ be a classical solution to $\L u=f$. Suppose that $f$ is Dini continuous. Then there exists a positive constant $c$, only depending on the operator $\L$, such that:
\begin{description}
 \item[{\it i)}]
 \begin{equation*}
\vert \partial^2 u(0,0) \vert \le c \left(\sup_{\Q_1(0,0)} \vert u \vert + \vert f(0,0) \vert +\int_{0}^1 \frac{\omega_{f}(r)}{r}dr \right);
\end{equation*}
 \item[{\it ii)}] for any points $(x,t)$ and $ (\x,\t) \in \Q_{\frac14}(0,0)$ we have
\begin{equation*}
|\partial^2 u(x,t)- \partial^2 u(\x, \t)| \le c\left(d \sup_{\Q_{1}(0,0)} |u| + d \sup_{\Q_{1}(0,0)}|f|  + \int_0 ^d \frac{\omega_{f}(r)}{r}dr+ d \int_{d} ^1 \frac{\omega_{f}(r)}{r^2}dr\right).
\end{equation*}
where $d: =d ( (x,t), (\x,\t))$ and $\partial^2$ stands either for $\partial^2_{x_i x_j}$, with $i,j=1,\ldots,m$, or for $Y$.
\end{description}
\end{theorem}
\begin{remark}
	The proof of the above statement is derived applying the techniques of \cite[Section 6]{PRS} combined with the proof of 
	the Dini continuity of the coefficients following the lines of Section \ref{proof2}, provided that \textbf{(H3A)} holds true.
\end{remark}

Lastly, we recall that the existence of a fundamental solution $\G$ for the operator $\L$ under the regularity assumptions 
\textbf{(C)} has widely been investigated over the years, and the Levy parametrix method provides us with a classic fundamental solution. Among the first results of this type we recall \cite{Weber}, \cite{IlIn} and \cite{Sonin} and we remark that this method was firstly considered 
in this setting by Polidoro in \cite{P}, and then later on extended in the works \cite{PDF, PODF}.
In particular, we report here the existence result of a classical fundamental solution for $\L$ proved in \cite[Theorem 1.4-1.5]{PDF}.
\begin{theorem}
	\label{ex-prop}
	Let us consider an operator $\L$ of the form \eqref{defL} under the assumptions 
	\textbf{(C)}. Then there exists a fundamental solution 
	$\G: \R^{N+1} \times \R^{N+1} \rightarrow \R$ for $\L$ with the following properties:
	\begin{enumerate}
	 	\item  $\G (\cdot, \cdot; \x, \t) \in L^{1}_{\loc} (\R^{N+1})
			\cap C(\R^{N+1} \setminus \{ ( \x, \t )\}$ for every $(\x, \t) \in \R^{N+1}$;
		\item $\G (\cdot, \cdot; \x, \t)$ is a classical solution of $\L u = 0$ in $\R^{N+1} \setminus \{ (\x, \t)  \}$
			for every $(\x, \t) \in \R^{N+1}$ in the sense of Definition \ref{solution};
		\item let $\phi \in C(\R^{N})$ such that for some positive constant $c_{0}$ we have
			\begin{equation} \label{cond1}
				| \phi (x) | \le c_{0} e^{c_{0} |x|^{2}} \qquad \text{for every  } \, x \in \R^{N},
			\end{equation}
			then there exists 
			\begin{equation} \label{cond2}
				\lim \limits_{(x,t) \to (x_{0}, \t) \atop t > \t} 
				\int \limits_{\R^{N}} \G (x,t; \x, \t) \phi(\x) d\x = \phi(x_{0})
				\qquad \text{for every  } \, x_{0} \in \R^{N};
			\end{equation}
		\item the reproduction property holds. Indeed, 
		for every $x, \x \in \R^{N}$ and $t,\tau \in \R$ with $\t < s < t$: 
		\begin{equation*}
			\G(x,t; \x,\t) = \int \limits_{\R^{N}} \G(x,t;y,s) \, \G(y,s; \x, \t) \, dy;
		\end{equation*}
		\item for every $(x,t), (\x, \t) \in \R^{N+1}$ with $t \le \t$
		we have that
			$\G (x,t; \x, \t) = 0$;
	\end{enumerate} 
	Moreover, if for every $i,j = 1, \ldots, m_{0}$ there exist the derivatives 
	 $\p_{x_{i}}a_{ij}, \p_{x_{i}}b$ $\in C^{\a}(\R^{N+1})$ and are bounded functions, then
	 there exists a fundamental solution 
	$\G^{*}$ to $\L^{*}$  verifying the dual properties of this statement and 
	\begin{equation*}
		\G^{*}(x, t; \x, \t) = \G (\x,\t; x,t) \qquad \text{for every  } \, \, (x,t), (\x, \t) \in \R^{N+1}, \, (x,t) \ne (\x, \t).	
	\end{equation*}
\end{theorem}

\setcounter{equation}{0}\setcounter{theorem}{0}
\section{Proof of Theorem \ref{gauss-bound-K}}
\label{proof1}
This section is devoted to the proof of Gaussian bounds for the fundamental solution defined in Definition \ref{funsolk}. 
We remark that all the other results in this section are obtained under the less restrictive assumptions \textbf{(H3B)}.
Now, let us begin with some preliminary results. Quite recently, there have been various developments in the study of weak solutions to $\L u =  0$, and in particular the following Harnack inequality \cite[Theorem 1.4]{AR-harnack} holds true in our setting.

\begin{theorem}[Harnack inequality]
	\label{harnack-thm}
	Let $\Q^0=B_{R_0}\times B_{R_0}\times \ldots \times B_{R_0} \times(-1,0] $ 
	and let $u$ be a non-negative weak solution to $\L u = 0$ in 
	$\O\supset \Q^0$ under the assumptions \textbf{(H1)}-\textbf{(H3B)}. 
	Then
	\begin{eqnarray}\label{harnackAR}
		\sup \limits_{\Q_{-}} u \, \le \, C   \inf_{\Q_+} u   ,
	\end{eqnarray}
	where $\Q_+=B_{\omega}\times B_{\omega^3}\times\ldots \times B_{\omega^{2\kappa+1}}\times 
	(-\omega^2,0]$
	and $\Q_-=B_{\omega}\times B_{\omega^3}\times\ldots \times B_{\omega^{2\kappa+1}}\times 
	(-1,-1+\omega^2]$. 
	Moreover, the constants $C$, $\omega$ and $R_0$ only depend on the homogeneous dimension 
	$Q$ defined in \eqref{hom-dim}, $q$ and on the ellipticity constants $\lambda$ and $\Lambda$ in 
	\eqref{hypcost}. 
\end{theorem}
\begin{remark}
	When considering the assumption \textbf{(H3A)} the constants appearing in the above statement do not depend on the $L^q$ norm of the coefficients of the operator $\L$,
	since we assume $|b(x,t)| \leq M$, $|c(x,t)| \leq M, \quad \forall (x,t) \in \R^{N+1}$. 
\end{remark}
We recall the following result, which will be useful in the proof of the upcoming Lemma \ref{lem-harnack}.
\begin{remark}\label{remarkLr}
Let $u$ be a weak solution to $\L u=0$ and $r>0$. Then $v:=u \circ \delta_\r$ solves equation $\L^{(r)}v=0$, where
\begin{eqnarray*}
\L^{(r)}v:=\div(A^{(r)} D u) - \div(a^{(r)}v) + \langle b^{(r)}, D v \rangle +  c^{(r)}v+\langle Bx, Dv \rangle -\p_t v
\end{eqnarray*}
with $A^{(r)}=A \circ \delta_r$, $a^{(r)}=r(a \circ \delta_r)$, $b^{(r)}=r(b \circ \delta_r)$ and $c^{(r)}=r^2(c \circ \delta_r)$.
\medskip\\
Moreover, if $u$ is a solution to $\L u=0$, then, for any $\zeta \in \R^{N+1}$, $v:=u \circ \ell_z$ solves equation $(\L\circ \ell_z)v=0$, where $\L\circ \ell_z$ is the operator obtained by $\L$ via a $\ell_z$-translation of the coefficients.
\end{remark}

For $\beta, r, R >0$ and $z_0 \in \R^{N+1}$, we define the cones
\begin{eqnarray*}
P_{\beta, r, R}:=\lbrace z \in \R^{N+1}:z=\delta_\r(\xi,\beta), |\xi|<r,0 < \r \leq R \rbrace,
\end{eqnarray*}
and we set $P_{\beta, r, R}(z_0):=z_0 \circ P_{\beta, r, R}$. We are now in a position to derive the following Lemma, which is a consequence of Theorem \ref{harnack-thm}.
\begin{lemma}\label{lem-harnack}
Let $z \in \R^{N+1}$ and $R \in (0,1]$. Moreover, let $u$ be a continuous and non-negative weak solution to $\L u =0$ in $\Q_R(z)$ under the assumptions \textbf{(H1)}-\textbf{(H3B)}. Then we have
\begin{equation*}
\sup_{P_{1,\omega,R/R_0}(z)}\leq C u(z),
\end{equation*}
where $C$, $R_0$ and $\omega$ are the constants appearing in Theorem \ref{harnack-thm} and they only depend on $Q$, $\lambda$, $\Lambda$ and $q$.
\end{lemma}
\begin{proof}
Let $w \in P_{1,\omega,R}(z)$, i.e. $w =z \circ \delta_\r(\xi, 1)$ for some $\r \in (0,R]$ and $|\xi|<\omega$. We now define the function $u_{z,\r}:=u \circ \ell_z \circ \delta_\r$, which is a continuous and non-negative solution to $\L^{(\r)}u_{z,\r}=0$ in $\Q_{R_0}(0,0) \subset \Q_{R/\r}(0,0)$ in virtue of Remark \ref{remarkLr}. Thus, we can apply the Harnack inequality \eqref{harnackAR} and infer
\begin{eqnarray*}
u(w)=u_{z,\r}(\xi,1)\leq \sup_{\Q_-}u_{z,\r}\leq C \inf_{\Q_+}u_{z,\r} \leq C u_{z,\r}(0,0)=C u(z).
\end{eqnarray*}
\end{proof}

We next state a global version of the Harnack inequality, which is a crucial step in proving the Gaussian lower bound (see Theorem \ref{lowerbound} below).
\begin{theorem}[Global Harnack inequality]\label{globalharnack}
Let $t_0 \in \R$ and $\tau \in (0,1]$. If $u$ is a continuous and non-negative weak solution to $\L u = 0$ in 
	$\R^{N+1}\times (\tau-t_0,\tau+t_0)$ under the assumptions \textbf{(H1)}-\textbf{(H3B)}, then we have
	\begin{eqnarray*}
	u(\xi,t) \leq c_0 e^{c_0 \langle C^{-1}(t-t_0)\left( \xi-e^{(t-t_0)B}x\right),\xi-e^{(t-t_0)B}x\rangle}
	u(x,t_0),\quad t \in (t_0,\tau+t_0),\quad x,\xi \in \R^{N},
	\end{eqnarray*}
where $C$ is the matrix introduced in \eqref{defC} and $c_0$ is a positive constant only depending on $Q$, $\lambda$, $\Lambda$ and $q$.
\end{theorem}
The proof of Theorem \ref{globalharnack} is based on a classical argument that makes use of the so-called Harnack chains, alongside with control theory. Moreover, the proof of this theorem follows the one of \cite[Theorem 3.6]{LPP}, with the only difference that we here apply Theorem \ref{harnack-thm} and Lemma \ref{lem-harnack} instead of Theorem 3.1 and Lemma 3.5 of \cite{LPP}. Indeed, the method we rely on has the advantage of highlighting the geometric structure of  the operator $\L$ and can be therefore automatically extended to more general operators. For this reason, we here do not show the derivation of Theorem \ref{globalharnack} and we refer the reader to \cite{LPP} for the details.

In proving our main result Theorem \ref{lowerbound}, we will also make use of the following result, which provides an upper bound for the fundamental solution. Les us remark that in the upcoming Theorem we consider operators of the form \eqref{Ltilde} satisfying the less restrictive assumptions \textbf{(H1)}-\textbf{(H3B)}.
\begin{theorem}[Gaussian upper bound]\label{upperbound}
Let $\tilde{\L}$ be an operator of the form \eqref{Ltilde} satisfying assumptions \textbf{(H1)}-\textbf{(H3B)}. Then there exists a positive constant $c_1$, only dependent on $Q$, $\lambda$, $\Lambda$ and $q$, such that
\begin{equation}\label{upperbound}
\Gamma(x,t;y,t_0)\leq \frac{c_1}{\left(t-t_0\right)^{\frac{Q}{2}}}\exp \left( -\frac{1}{c_1}\left| \delta^0_{(t-t_0)^{-\frac{1}{2}}}\left( y-e^{(t-t_0)B}x\right)\right|^2\right)
\end{equation}
for any $0 < t-t_0 \leq 1$ and $x, y \in \R^{N}$.
\end{theorem}
\begin{proof}
The Gaussian upper bound \eqref{upperbound} was proved in \cite[Theorem 4.1]{LPP} under the stronger assumption that the coefficients $a_i$, $b_i$, with $i=1,\ldots,m_0$, and $c$ are bounded measurable functions of $(x,t)$ and in \cite[Theorem 1.4]{ALP} under the additional hypothesis that the coefficients $b_i$, with $i=1,\ldots,m_0$, are null. The more general case where the first order coefficients satisfy assumption \textbf{(H3B)} can be treated similarly. 
Thus, here we just sketch the few adjustments required to adapt the proof of \cite{LPP} to the present case.

The proof relies on the combination of a Caccioppoli type inequality and a Sobolev type inequality. In order to handle the more general case, we need to replace the Caccioppoli and the Sobolev inequality contained in \cite{ALP} (Theorem 2.3 and 2.5, respectively) with the ones given in \cite{AR-harnack}. More precisely, we consider the Caccioppoli inequality \cite[Theorem 3.4]{AR-harnack} and we focus on the new term involving the coefficient  $a \in \left( L^q_{\loc} (\O) \right)^{m_0}$, which is handled as follows
\begin{align*}
&\frac{(2p-1)}{2} \, \int_{\Q_r} \langle  a, D_{m_0}  v^2 \rangle \psi^2 \,  + 2 p \, \int_{\Q_r} \langle a v^2, D_{m_0} \psi  \rangle \psi	\\ \nonumber		
			&=-\frac{2p-1}{2}\int_{\Q_r} \div \cdot a \, v^2 \psi^2-(2p-1)\int_{\Q_r}\langle a, D_{m_0}\psi \rangle \psi v^2+ 2 p \, \int_{\Q_r} \langle a v^2, D_{m_0} \psi  \rangle \psi \\ \nonumber 
			&\leq|2p-1| \, \int_{\Q_r} |a| |\psi| |D_{m_0}  \psi| v^2
			+ 2 |p| \, \int_{\Q_r} |a| |D_{m_0} \psi| |\psi|  v^2 \\ \nonumber 
		& \le
		\frac{C|2p-1|}{(r-\r)} \parallel a \parallel_{L^q(\Q_r)} \, \parallel v \parallel_{L^{2\b}(\Q_r)}^2 +
		 \frac{|2p| \, c_1}{(r - \r)}  \, \parallel a \parallel_{L^{q}(\Q_r)} \, \parallel v \parallel_{L^{2\b}(\Q_r)}^2,
	\end{align*}
where $\b=\frac{q}{q-1}$, and $q$ is the integrability exponent introduced in \textbf{(H3B)}.
From this point, we obtain the Caccioppoli inequality reasoning as in the proof of \cite[Theorem 3.4]{AR-harnack}.

As far as the Sobolev inequality is concerned, we find two extra terms in the representation formula of sub-solutions. More precisely, following the notation of \cite[Theorem 3.3]{AR-harnack}, the term $I_0(z)$ here becomes
\begin{equation*}
I_0 (z) = \int_{\Q_r} \left[ -\langle a, D ( \psi \G(z, \cdot)) \rangle  v 
		\right](\z) d \z \hspace{1mm}
		+ \int_{\Q_r} \left[ \langle b, D v
		\rangle \G (z, \cdot) \psi \right](\z)d\z \hspace{1mm} + \,\int_{\Q_r} \left[ c v \G (z, \cdot) \psi \right](\z) d \z.
\end{equation*}
Since
\begin{equation*}
		\langle a, D v \rangle \,  \in L^{2 \frac{q}{q+2}} \hspace{4mm} 
		\text{for } a  \in L^q, \hspace{2mm} q > \frac{Q+2}{2} \hspace{2mm}
		\text{and } v \in L^2,
	\end{equation*}
reasoning as in \cite[Theorem 3.3]{AR-harnack} we infer
\begin{align*}
	\parallel I_0 (\z) \parallel_{L^{2\a}(\Q_\r)} 	&\leq 
		\parallel \G * \left( \langle a,  D_{m_0}v \rangle \psi \right) + \G * \left( \langle b,  D_{m_0}v \rangle \psi \right)  + \G * \left( c v \psi \right) 
		\parallel_{L^{2 \alpha}(\Q_\r)}\\
		&\le C \cdot \left( \parallel a \parallel_{L^q (\Q_\r)} + \parallel b \parallel_{L^q (\Q_\r)}  \parallel D_{m_0}v \parallel_{L^2 (\Q_\r)}
		+  \parallel c \parallel_{L^q (\Q_\r)} \parallel v \parallel_{L^2 (\Q_\r)})\right),
	\end{align*}
where 
\begin{equation*}\label{def-alpha}
	\alpha=\frac{q(Q+2)}{q(Q-2)+2(Q+2)}.
\end{equation*}
In addition, the term $I_3(z)$ here becomes
\begin{align*}
I_3 (z) &= \int_{\Q_r} \left[ \langle A Dv, D( \G (z, \cdot ) \psi)
		\rangle \right](\z) d \z \hspace{1mm} - \hspace{1mm}
		\int_{\Q_r} \left[ \left( \G (z, \cdot ) \psi \right) Yv
		\right](\z) d \z \hspace{1mm}  \\
		&\hspace{4mm} + \, \int_{\Q_r} \left[ \langle a, D (\G (z, \cdot) \psi ) 
		\rangle v \right](\z) d \z  \, - 
		 \int_{\Q_r} \left[ \langle b, D v
		\rangle \G (z, \cdot) \psi \right](\z) d \z - \, \int_{\Q_r} \left[ c v \G (z, \cdot) \psi \right](\z) d \z 
\end{align*}
and can be treated exactly as the analogous one in \cite{AR-harnack}.
The rest of the proof of the Sobolev inequality follows the one contained in \cite[Theorem 3.3]{AR-harnack}.

Lastly, when considering our case, the proof of inequality (3.4) in \cite[Theorem 3.3]{ALP} needs to be treated slightly differently. 
In particular, inequality (3.2) in \cite{ALP} becomes 
\begin{align*}
&\int_{\R^N}u^2\gamma_R^2e^{2h}u^2 \Big|_{t=\tau}dx-2\int\int_{[\tau,\eta]\times{\R^N}}e^{2h}u^2\gamma_R^2 (3 \langle AD_{m_0}h,D_{m_0} h \rangle-Yh-\\
&\quad\quad\qquad\qquad\quad\qquad-\tr B +\langle a, D_{m_0} h \rangle + \langle b, D_{m_0} h \rangle )dxdt 
\\ &\leq \int_{\R^N}u^2\gamma_R^2e^{2h}u^2 \Big|_{t=\eta}dx +2\int\int_{[\tau,\eta]\times{\R^N}}e^{2h}u^2(3\Lambda|D_{m_0}\gamma_R|+|Y\gamma_R|^2+\\&\quad\quad\qquad\qquad\qquad\quad - \tr B +\langle a, D_{m_0}\gamma_R \rangle \gamma_R+\langle b, D_{m_0}\gamma_R \rangle \gamma_R)dxdt.
\end{align*}
Thus, inequality (3.4) in \cite{ALP} can be rewritten as
\begin{align}\label{inequalityLP}
3 \langle AD_{m_0}h,D_{m_0} h \rangle-Yh-\tr B +\langle a, D_{m_0} h \rangle + \langle b, D_{m_0} h \rangle\leq 0, \qquad \eta -\frac{\eta-s}{k} \leq t \leq \eta,x \in \R^N,
\end{align}
where $h(x,t)=-\frac{|x|^2}{\delta}+\alpha(\eta-t)$, with $\delta=2(\eta-s)-k(\eta-t)$. Hence, inequality \eqref{inequalityLP} becomes
\begin{align*}
&3 \langle AD_{m_0}h,D_{m_0} h \rangle-Yh-\tr B +\langle a, D_{m_0} h \rangle + \langle b, D_{m_0} h \rangle\\ 
&\quad\leq \frac{12\Lambda |x|^2}{\delta^2}+\frac{2\Vert B\Vert |x|^2}{\delta}-\frac{k|x|^2}{\delta^2}-\alpha-\tr B+\frac{2}{\delta}\left(\Vert a \Vert_2^2+\Vert b \Vert_2^2 \right)+\frac{4|x|^2}{\delta}\\
&\quad \leq \frac{|x|^2}{\delta^2}\left(12\Lambda+2\delta\Vert B\Vert-k+4\delta\right)-\alpha-\tr B+\frac{2}{\delta}\left(\Vert a \Vert_2^2+\Vert b \Vert_2^2 \right)\\
&\quad \leq \frac{|x|^2}{\delta^2}\left(12\Lambda+4\Vert B\Vert-k+8\right)-\alpha-\tr B+\frac{2}{\delta}\left(\Vert a \Vert_2^2+\Vert b \Vert_2^2 \right)
\end{align*}
and therefore inequality \eqref{inequalityLP} holds true provided that we choose $\alpha=\frac{2}{\delta}\left(\Vert a \Vert_2^2+\Vert b \Vert_2^2 \right)-\tr B$ and $k$ big enough.
The rest of the proof follows the same lines of the one of \cite[Theorem 4.1]{LPP}.
\end{proof}
\begin{lemma}\label{boundDR}
Let $\L$ be an operator of the form \eqref{defL} satisfying assumptions \textbf{(H1)}-\textbf{(H3B)}. Then there exist two positive constants $R$ and $c_2$, only dependent on $Q$ and $B$, such that
\begin{equation}
\int_{\left|\delta_{(\sqrt{t-t_0})}^0\left( y-e^{(t-t_0)B}x\right)\right|\leq R}\Gamma(x,t;y,t_0)dx \geq c_2, \qquad 0<t-t_0\leq 1, y \in \R^N.
\end{equation}
\end{lemma}
\begin{proof}
We first notice that for a small enough constant $c_3$, which depends only on $Q$ and $B$, the function
\begin{equation*}
v(y,t_0):=\int_{\R^N} \Gamma (x,t;y,t_0) - e^{-c_3(t-t_0)}, \quad t>t_0,\quad y \in \R^N,
\end{equation*}
is a weak super-solution of the Cauchy problem
\begin{equation*}
			\begin{cases}
				\L^* v(y,t_0) = - e^{-c(t-t_0)}(c-\tr(B)+c_3)\leq 0, \qquad &t>t_0,\quad y \in \R^N , \\
v(y,t)=0, \qquad  &y \in \R^{N},
			\end{cases}
		\end{equation*}
where $\L^*$ is the adjoint operator defined in \eqref{Ladj}. Hence, in virtue of the maximum principle we infer $v \geq 0$, that is
\begin{equation}\label{intGamma}
\int_{\R^N} \Gamma (x,t;y,t_0) \geq e^{-c_3(t-t_0)}, \quad t>t_0,\quad y \in \R^N.
\end{equation}
We now observe that
\begin{align}\label{ineGamma2} \nonumber
&\int_{\left|\delta_{(\sqrt{t-t_0})}^0\left( y-e^{(t-t_0)B}x\right)\right|\geq R}\Gamma(x,t;y,t_0)dx \\&\quad\leq \frac{c_1}{\left(t-t_0\right)^{\frac{Q}{2}}}\int_{\left|\delta_{(\sqrt{t-t_0})}^0\left( y-e^{(t-t_0)B}x\right)\right|\geq R}\exp \left( -\frac{1}{c_1}\left| \delta^0_{(t-t_0)^{-\frac{1}{2}}}\left( y-e^{(t-t_0)B}x\right)\right|^2\right)dx\\ \nonumber
&\quad=c_1\int_{|z|\geq R}\exp\left(-\frac{1}{c_1}|z|^2 \right) dz,
\end{align}
where in the second line we have used the upper bound \eqref{upperbound} and in the third line we have performed the change of variables $z=\delta_{(\sqrt{t-t_0})}^0\left( y-e^{(t-t_0)B}x\right)$. Combining \eqref{intGamma} and \eqref{ineGamma2} and choosing $c_3$ small enough we obtain the thesis.
\end{proof}
We are now in a position to state and prove the following result concerning the Gaussian lower bound of the fundamental solution.
\begin{theorem}[Gaussian lower bound]\label{lowerbound}
Let $\L$ be an operator of the form \eqref{defL} satisfying assumptions \textbf{(H1)}-\textbf{(H3B)}. Then there exists a positive constant $c_4$, only dependent on $Q$, $\lambda$, $\Lambda$ and $q$, such that
\begin{equation}\label{lowerboundform}
\Gamma(x,t;y,t_0)\geq \frac{c_4}{\left(t-t_0\right)^{\frac{Q}{2}}}e^{-c_4 \langle C^{-1}(t-t_0)\left( y-e^{(t-t_0)B}x\right),y-e^{(t-t_0)B}x\rangle}
\end{equation}
for any $0 < t-t_0 \leq 1$ and $x, y \in \R^{N}$.
\end{theorem}
\begin{proof}
We restrict ourselves to the case where $x=0$, as the general statement can be obtained from the dilation and translation-invariance of the operator $\L$.
Then, for every $y \in \R^N$ and $R>0$, we set
\begin{equation*}
D_R:=\Big\lbrace \xi \in \R^N: \left|\delta_{\sqrt{\tau}}^0(y-e^{\tau B}\xi)\right|\leq R \Big\rbrace
\end{equation*}
and we compute
\begin{align*}
\meas(D_R)=\int_{D_R}d\xi &=R^Q \int_{ \left|\delta_{\sqrt{\tau}}^0(y-e^{\tau B}\xi)\right|\leq 1 }d\xi =R^Q\tau^{Q/2}\int_{ \left|(y-e^{\tau B}\xi)\right|\leq 1 }d\xi\\
&=R^Q\tau^{Q/2}\int_{ \left|\xi\right|\leq 1 }d\xi
=R^Q\tau^{Q/2}\meas(B_1(0))=c_5\tau^{Q/2},
\end{align*}
where the constant $c_5$ only depends on $B$ and $R$.
We also note that the function $\langle C^{-1}(t-t_0)\left( y-e^{(t-t_0)B}x\right),y-e^{(t-t_0)B}x\rangle$ is bounded by a constant $M$ in $D_R$ (see \cite[Lemma 3.3]{LP}). 
Lastly, we now set $\tau=\frac{t-t_0}{2}$ and apply to $\Gamma$ the global Harnack inequality stated in Theorem \ref{globalharnack}, which yields
	\begin{eqnarray}\label{globharnackGamma}
	\Gamma(y,t;y,t_0) \geq c_0 e^{-c_0 \langle C^{-1}(\tau)\left( \xi-e^{\tau B}x\right),\xi-e^{ \tau B}x\rangle}\Gamma(\xi,t+\tau;y,t_0),\quad y,\xi \in \R^{N}.
	\end{eqnarray}
	Hence, integrating inequality \eqref{globharnackGamma} over $D_R$, we infer 
\begin{align*}
\Gamma(y,t;y,t_0)&=\frac{c_6}{\tau^{Q/2}}\int_{D_R}\Gamma(y,t;y,t_0)d\xi \\&\geq 	\frac{c_6\,c_0}{\tau^{Q/2}}\int_{D_R}e^{-c_0 \langle C^{-1}(\tau)\left( \xi-e^{\tau B}x\right),\xi-e^{ \tau B}x\rangle}\Gamma(\xi,t+\tau;y,t_0)d\xi\\
&\geq \frac{c_6\,c_0}{\tau^{Q/2}}\int_{D_R} e^{-M}\Gamma(\xi,t+\tau;y,t_0)d\xi\\
&\geq \frac{c_7}{(t-t_0)^{Q/2}},
	\end{align*}
where the last inequality is a direct consequence of Lemma \ref{boundDR} and the constant $c_7$ only depends on $Q$, $\lambda$, $\Lambda$, $q$ and $B$.\\
Setting $\tau=\frac{3}{4}(t-t_0)$ and $x=0$, we apply once again Theorem \ref{globalharnack} and we get
\begin{align*}
\Gamma(0,t;y,t_0)&\geq c_0e^{-c_0\langle C^{-1}(\tau)y,y\rangle}\Gamma(y,t+\tau;y,t_0)\\
&\geq \frac{c_8}{(t-t_0)^{Q/2}}e^{-c_0\langle C^{-1}(\tau)y,y\rangle}
\geq  \frac{c_9}{(t-t_0)^{Q/2}}e^{-c_9\langle C^{-1}(t-t_0)y,y\rangle},
\end{align*}
where the last inequality is a consequence of a property of the covariance matrix $C$ (see \cite[Remark 4.5]{LPP}). This concludes the proof.
\end{proof}

\setcounter{equation}{0}\setcounter{theorem}{0}
\section{Proof of Theorem \ref{main1}}
\label{proof2}
Let us consider the operator $\L$ under the assumption \textbf{(H1)}-\textbf{(H3A)}.
Our first aim is to build a sequence of operators $\left( \L_\e \right)_{\e}$ satisfying the assumption \textbf{(C)}
of Theorem \ref{ex-prop}. Without loss of generality we restrict ourselves to the case of 
$(T_0, T_1) = (0,T)$, with $T > 1$.
Thus, we may consider $\r \in C^\infty_0 (\R)$ and $\psi \in C^\infty_0(\R^N)$ such that
\begin{align*}
	\int \limits_\R \rho(t) \, dt = 1& , \quad  \text{supp} \, \rho \subset B \left( \frac{T}{2},\frac{T}{4} \right), \quad  \text{and} \quad
	\int \limits_{\R^N} \psi (x) \, dx = 1 , \quad \text{supp} \, \psi \subset B(0,1),
\end{align*}
where by abuse of notation $B(\tfrac{T}{2},\tfrac{T}{4})$ denotes the Euclidean ball on $\R$ of radius 
$\tfrac{T}{4}$ and center $\tfrac{T}{2}$ of suitable dimension
and $B(0,1)$ denotes the Euclidean ball of $\R^N$ of radius $1$ and center $0$ of suitable dimension.
Then, for every $\e \in (0, 1]$ we classically construct two families of mollifiers 
\begin{align*}
	\r_\e(t) = \frac1\e \,  \r \left( \frac{t}{\e} \right), \qquad \psi_\e (x) = \frac{1}{\e^N} \, \psi \left( \frac{x}{\e} \right).
\end{align*}
Lastly, for every $\e \in (0, 1]$, for every $t \in \, (0,T)$ and $x \in \R^N$ we define
\begin{align*} 
    (a_{ij})_{\e}(x,t) \, &:= \int \limits_\R \int \limits_{\R^N} a_{ij}(x-y, (1-\e)t  + \tau) \psi_\e(y) \, \rho_\e (\tau) \, dy \, d\tau , \quad i,j=1, \ldots, N, \\ \nonumber
    (b_i)_{\e}(x,t) \, &:= \int \limits_\R \int \limits_{\R^N} b_i(x-y, (1-\e)t  + \tau) \psi_\e(y) \, \rho_\e (\tau) \, dy \, d\tau,  \, \quad \quad i=1, \ldots, N, \\ \nonumber
    c_\e (x,t) &:= \int \limits_\R \int \limits_{\R^N} c(x-y, (1-\e)t  + \tau) \psi_\e(y) \, \rho_\e (\tau) \, dy \, d\tau.
\end{align*}
These newly defined coefficients are smooth and bounded from above by the same constant appearing in assumption \textbf{(H3A)}.
Moreover, for every $\e \in (0,1]$ the coefficients $(a_{ij})_{\e}$ are smooth and satisfy the following version of 
\textbf{(C)}
\begin{align*}
	| (a_{ij})_{\e} | \le \sup_{(x,t) \in \R^{N+1}} |a_{ij}(x,t)| \le \, \Lambda,
\end{align*}
\begin{align*}
\Big \lvert \tfrac{\p (a_{ij})_\e}{\p x_k} (x,t) \Big \rvert &\le \Big\lvert \int_{B(0,\e)}\Big\vert \tfrac{\p \psi_\e}{\p x_k}(y) \Big\vert \, \vert a_{ij}(x-y,(1-\e)t+\tau)dy \vert 
	\Big\rvert \\
	&\le \frac{1}{\e^N}\, \e^N\, M \sup_{x\in \R^N}\Big\vert \tfrac{\p \psi_\e}{\p x_k}(x) \Big\vert=M \sup_{x\in \R^N}\Big\vert \tfrac{\p \psi_\e}{\p x_k}(x) \Big\vert < +\infty
\end{align*}
for every $i,j=1,\ldots,N$ and for every $k=1,\ldots, m_0$.
The same statement holds true for the coefficients $c_\e$ and $(b_{i})_\e$, with $i=1, \ldots, N$ and $\e \in (0,1]$. Thus, thanks to the mean value theorem along the direction of the vector fields $\tfrac{\p }{\p x_k}$, the coefficients $(a_{ij})_{\e}$, $c_\e$ and $(b_{i})_\e$, with $i=1, \ldots, N$ and $\e \in (0,1]$, are Lipschitz continuous and therefore Dini continuous. Hence, we can apply Theorem \ref{th-1} to $(\GL^{\!\!\! \e})_{\e}$ for every $\e \in (0,1]$. Thus, there exists a sequence of 
equibounded fundamental solutions $(\GL^{\!\!\! \e})_{\e}$, in the sense that each of them satisfies Theorem \ref{ex-prop},
i.e. for every $(x,t) , (\x, \tau) \in \R^{N+1}$, with $0 < t < \tau < T$
 \begin{align*}
		C^{-} \, \GK^{\! \! \! \! \l^{-}} (x,t; y,\t) \, \le \, \GL^{\!\!\! \e}(x,t; y,\t) \, \le \,
		C^{+} \, \GK^{\! \! \! \! \l^{+}} (x,t;y,\t).
\end{align*}
We point out that, since the coefficients of $\L_\e$ are uniformly bounded by $\Lambda$, the coefficients of Theorem \ref{main1} 
do not depend on $\e$. Moreover, given property $(5)$ of Theorem \ref{ex-prop}, for every $\e \in (0,1]$ we have 
$\GL^{\!\!\! \e} (x,t; \x, \t) = 0$ whenever $t \le \tau$.

\medskip

As a first step, we prove there exists a converging subsequence $(\GL^{\!\!\! \e})_{\e}$ in every compact subset of 
$\{ (x,t; \x, \t) \in \left( \R^{N} \times (T, + \infty) \right)^2 \mid (x,t) \ne (\x, \t) \}$. For this reason, we define a sequence of open subsets 
$(\O_p)_{p \in \N}$ of $\R^{2N+2}$
\begin{align*}
	\O_p := &\left\{ x, \x \in \R^{N} :  \, |x|^2 < p^2, |\x|^2 < p^2, |x-\x|^2  > \frac{1}{2p} \right\}
			\times \left\{ t,\t \in \Big ( \frac1p, T - \frac1p \Big )  : \, \, |t-\t|^2 > \frac{1}{2p} \right\}.
\end{align*}	
Note that $\O_p \subset \subset \O_{p+1}$ for every $p \in \N$. 
Moreover, $\cup_{p=1}^{+\infty} \O_p = \{ (x,t; \x, \t) \in \left( \R^{N} \times (0, T) \right)^2 : (x,t) \ne (\x, \t) \}$.
Since $\GL^{\! \! \! \! \l^{+}}$ is a bounded function in $\O_p$, 
we have that $( \GL^{\!\!\! \e})_{\e}$ is an equibounded sequence in every
$\O_p$. Then, as the sequence $( \GL^{\!\!\! \e})_{\e}$ is equibounded in $\O_2$,
it is equicontinuous in $\O_{1}$ thanks to Theorem \ref{th-1}. 
Moreover, by Theorem \ref{ex-prop} and Theorem \ref{th-1}, we also have that 
\begin{align*}
	\left(\tfrac{\p \GL^{\!\!\! \e}}{\p x} \right)_{\e}, \quad 
	\left(\tfrac{\p \GL^{\!\!\! \e}}{\p \x} \right)_{\e}, \quad
	\left(\tfrac{\p^{2} \GL^{\!\!\! \e}}{\p x^{2}} \right)_{\e}, \quad 
	\left(\tfrac{\p^{2} \GL^{\!\!\! \e}}{\p \x^{2}} \right)_{\e}, \quad 
	\left(Y \GL^{\!\!\! \e} \right)_{\e}, \quad 
	\text{and} \quad \left(Y^*_{(\x,\t)} \GL^{\!\!\! n} \right)_{\e} 
\end{align*}
are bounded sequences in $C^{0}(\O_{1})$. Here $Y$ is the Lie derivative defined in \eqref{lie-diff} and $Y^*_{(\x,\t)}$ is its adjoint, computed with respect to the variable $(\x,\t)$. Thus, there exists a subsequence $(\GL^{\!\!\! 1,\e_1})_{\e_1}$ that converges uniformly to some function $\G_{1}$ that satisfies \eqref{gauss-bound-K} in $\O_1$.
Moreover, $\G_{1} \in C^{2}(\O_{1})$ and, for every $(x_0,t_0) \in \R^N \times (0,T)$ such that $x_0^2  + t_0^2 <1$ the function $u(x,t) := \G_{1}(x,t; x_0,t_0)$ is a classical solution to $\L u = 0$ in the set $\big\{ (x,t) \in \R^N \times (0,T) \, \mid (x,t; x_0,t_0) \in \O_{1} \big\}$.

We next apply the same argument to the sequence $(\GL^{\!\!\! 1,\e_1})_{\e_2}$ on the set $\O_{2}$, and
obtain a subsequence $(\GL^{\!\!\! 2,\e_2})_{\e_2}$ that converges in $C^{2}(\O_{2})$ to some function $\G_{2}$, that belongs to $C^{2}(\O_{2})$ and satisfies the bounds \eqref{gauss-bound-K} in $\O_2$. Moreover, for every $(x_0,t_0) \in \R^N \times (0,T)$ such that $x_0^2 + t_0^2 <4$ the function $u(x,t) := \G_{2}(x,t; x_0,t_0)$ is a classical solution to $\L u = 0$ in the set 
$\big\{ (x,t) \in \R^N \times (0,T) \mid (x,t; x_0,t_0) \in \O_{2} \big\}$.

We next proceed by induction. Let us assume that the sequence $(\GL^{\!\! q-1,\e_{q-1}})_{\e_{q-1}}$ on the set 
$\O_{q}$ has been defined for some $q \in \N$. We extract from it a subsequence $(\GL^{\!\! q,\e_q})_{\e_q}$ converging in 
$C^{2}(\O_{q})$ to some function $\G_{q}$, satisfying \eqref{gauss-bound-K} in $\O_q$ and it agrees with $\G_{q-1}$ on the set $\O_{q-1}$.

Next, we define a function $\GL$ in the following way: for every $(x,t), (\x,\t)  \in \R^{N} \times (0,T)$ with 
$(x,t) \ne (\x,\t)$ we choose $q \in \N$ such that $(x,t; \x,\t) \in \O_{q}$ and we set $\GL(x,t;\x,\t) := \G_{q} (x,t;\x,\t)$. 
This argument provides us with a non ambiguous definition of $\GL$. Indeed, if $(x,t) \in \O_{p}$, then $\G_{p}(x,t;\x,\t) = \G_{q}(x,t;\x,\t)$.

We next check that $\GL$ has the properties listed in the statement of the Theorem \ref{main1}. As every $\GL^{\!\!\! \e}(x,t; x_{0}, t_{0})=0$ whenever $t \le t_{0}$, also $\GL(x,t; x_{0}, t_{0}) = 0$ whenever $t \le t_{0}$. 
For the same reason, it satisfies Theorem \ref{gauss-bound-K}. 
Moreover, for every $(x_0,t_0) \in \R^N \times (0,T)$, $(x,t) \mapsto \GL(x,t; x_0,t_0) \in L^{1}_{\loc}(\R^{N} \times (0,T)) 
\cap C^{2}_{\loc}(\R^{N} \times (0,T) \setminus \{ (x_{0},  t_{0})\})$, and is a weak solution to $\L u = 0$ in 
$\R^{N} \times (0,T) \setminus \{ (x_{0}, t_{0})\}$. 

As far as we are concerned with the reproduction property \emph{1.} of Theorem \ref{main1},
we use the upper bound in \eqref{gauss-bound-K}, which yields
\begin{equation*}
	\GL^{\!\!\!\e} (x,t;  \x,\t ) \GL^{\!\!\!\e} (\x,\t;  x_{0}, t_{0})  \le 
	C^{+} \, \GL^{\! \! \! \! \l^{+}} (x,t; \x,\t ) C^{+} \, \GL^{\! \! \! \! \l^{+}} (\x,\t;  x_{0}, t_{0}),
\end{equation*}
and the reproduction property 
\begin{equation*}
    \int_{\R^N} \GL^{\! \! \! \! \l^{+}} (x,t; \x,\t ) 
	\GL^{\! \! \! \! \l^{+}} (\x ,\t;  x_{0}, t_{0}) \, d\x \, d \t 
	= \GL^{\! \! \! \! \l^{+}} (x,t; x_{0}, t_{0}) \, < + \infty,
\end{equation*}
which allows us to use the Lebesgue convergence theorem. Thus the property holds true.

To conclude the proof of Theorem \ref{main1} we have to verify that for any bounded function $\phi \in C (\R^{N})$ and any $x,y \in \R^{N}$
 the function
	\begin{equation} \label{eq-repr-sol}
		u(x,t) \, = \, \int \limits_{\R^{n}} \GL(x,t; \x, t_{0}) \, \phi(\x) \, d\x 
	\end{equation}
is such that  
			\begin{align} \label{weak-cauchy}
				\begin{cases}
					\L u (x,t) = 0 \qquad \qquad &(x,t) \in \R^N \times [0,T], \\
					\lim \limits_{(x,t) \to (y,T) \atop t >T} u(x,t) = \phi(y) \qquad &y \in \R^N.
				\end{cases}
			\end{align} 
By the usual standard argument, we differentiate under the integral sign
\begin{equation*}
		\L u(x,y,t) \, = \, \int \limits_{\R^{N}} \L \GL(x,t; \x, t_{0}) \, \phi(\x) \, d\x \, = \, 0.
\end{equation*}

Thus, we are left with the proof of \eqref{uv-def}. Thanks to Theorem \ref{ex-prop} applied to
the regularized operator $\L_\e$, we have that for every $\e \in (0,1]$ and for every $y \in \R^N$ 
the following holds
 \begin{align*}
 	\lim \limits_{(x,t) \to (y, t_0) \atop t >t_0} u_\e(x,t) = \phi(y) ,
	\quad \text{where } u_\e(x,t) := \int_{\R^{N}} \GL^{\! \!\e}(x,t;y,t_0) \, \phi(y) \, dy.
 \end{align*}
Now, thanks to Theorem \ref{gauss-bound-K} we are able to apply the dominated convergence theorem, and thus		
for every $y \in \R^N$ we have 	
\begin{align*}
	\phi (y) = \lim \limits_{\e \to 0} \lim \limits_{(x,t) \to (y, t_0) \atop t >t_0} u_\e(x,t) =
		      \lim \limits_{(x,t) \to (y, t_0) \atop t >t_0} u(x,t) .
\end{align*}	
We conclude the proof by adapting the same argument when considering the function $v$.
$\hfill \square$

\end{document}